\documentclass[11pt, reqno]{amsart}
\usepackage{amssymb}
\usepackage{verbatim}
\usepackage[vcentermath]{youngtab}
\usepackage{float}
\usepackage{hyperref}
\usepackage{tikz}
\usepackage{enumerate}
\usepackage[left=1.1in, right=1.1in, top=1in, bottom=.9in]{geometry}
\usepackage{caption}
\usepackage{subcaption}

\usepackage{times}
\usepackage[T1]{fontenc}
\usepackage{mathrsfs}
\usepackage{epsfig}
\usepackage{color}
\usepackage{array} 
\newcolumntype{L}{>{$}l<{$}}
\newcolumntype{R}{>{$}r<{$}}

\usepackage{enumitem}

\newtheorem{thm}{Theorem}[section]
\newtheorem{lemma}[thm]{Lemma}

\newtheorem{cor}[thm]{Corollary}
\newtheorem{prop}[thm]{Proposition}

\theoremstyle{remark}

\newtheorem*{remark}{Remark}

\theoremstyle{definition}

\newtheorem{question}[thm]{Question}
\newtheorem{example}[thm]{Example}

\def\ZZ{\mathbb{Z}}
\def\QQ{\mathbb{Q}}
\def\PP{\mathbb{P}}

\def\RR{\mathbb{R}}
\def\CC{\mathbb{C}}

\def\wt{\widetilde}

\def\preper{\mathrm{PrePer}}

\def\ve{\varepsilon}
\def\I{\mathcal{I}}

\def\vol{\mathrm{Vol}}

\def\multi#1#2{\ensuremath{\left(\kern-.3em\left(\genfrac{}{}{0pt}{}{#1}{#2}\right)\kern-.3em\right)}}

\numberwithin{equation}{section}

\begin{document}

\title{Polynomials with many rational preperiodic points}

\author{John R. Doyle}
\address{Dept. of Mathematics\\
Oklahoma State University\\
Stillwater, OK 74078}
\email{john.r.doyle@okstate.edu}

\author{Trevor Hyde}
\address{Dept. of Mathematics\\
University of Chicago \\
Chicago, IL 60637}
\email{tghyde@uchicago.edu}

\begin{abstract}
In this paper we study two questions related to exceptional behavior of preperiodic points of polynomials in $\QQ[x]$. 
We show that for all $d\geq 2$, there exists a polynomial $f_d(x) \in \QQ[x]$ with $2\leq \deg(f_d) \leq d$ such that $f_d(x)$ has at least $d + \lfloor \log_2(d)\rfloor$ rational preperiodic points. Furthermore, we show that for infinitely many integers $d$, the polynomials $f_d(x)$ and $f_d(x) + 1$ have at least $d^2 + d\lfloor \log_2(d)\rfloor - 2d + 1$ common complex preperiodic points.
\end{abstract}
\maketitle

\section{Introduction}
Let $K$ be a field and let $f(x) \in K[x]$ be a polynomial. We write $f^n(x)$ to denote the $n$-fold composition of $f$ with itself. A point $\alpha \in \overline{K}$ is \emph{preperiodic} under $f(x)$ if the orbit $\{f^n(\alpha) : n\geq 0\}$ is finite. 
Let $\preper(f,K)$ denote the set of $K$-rational preperiodic points of $f$,
\[
    \preper(f,K) := \{\alpha \in K : \alpha \text{ is preperiodic under } f\}.
\]

The following questions arise naturally in arithmetic and complex dynamics:
\begin{question}
How many rational preperiodic points can a degree-$d$ polynomial $f(x) \in \QQ[x]$ have?
\end{question}

\begin{question}
How many complex preperiodic points can degree-$d$ polynomials $f(x), g(x) \in \CC[x]$ have in common?
\end{question}

Both questions are conjectured to have answers in the form of uniform upper bounds depending only on $d$ (subject to some minor caveats described below). 
Our main result proves the existence of a sequence of polynomials $f_d(x) \in \QQ[x]$ of degree at most $d$ which simultaneously exhibit extremal behavior for both questions: $f_d(x)$ has many rational preperiodic points, and the polynomials $f_d(x) + i $ and $f_d(x) + j$ have many common complex preperiodic points for small integers $i$ and $j$.

\begin{thm}
\label{thm main intro}
For all integers $d\geq 2$ there exists a polynomial $f_d(x) \in \QQ[x]$ such that $2\leq \deg(f) \leq d$ and
\begin{enumerate}
    \item $f_d(x)$ has at least $d + \lfloor \log_2(d)\rfloor$ rational preperiodic points,
    
    \item for all $0 \le i < j \le \log_2(d)$,
    \[
        \Big|\preper(f_{d}(x)+i,\CC) \cap \preper(f_{d}(x)+j,\CC)\Big| < \infty,
    \]

    \item and
    \[
        \left|\bigcap_{i=0}^{\lfloor \log_2(d)\rfloor} \preper(f_{d}(x) + i,\CC)
        \right| \ge \deg(f_{d})(d - 1) + 1.
    \]
\end{enumerate}
\end{thm}

\begin{remark}
Using Lagrange interpolation one may easily construct degree-$d$ polynomials with $d + 1$ rational preperiodic points. Each rational preperiodic point beyond $d + 1$ imposes an additional constraint. Theorem \ref{thm main intro} shows that it is possible to get an improvement on the order of (at least) $\log(d)$ on the Lagrange interpolation construction.
\end{remark}

Given an integer $d\geq 2$, let
\begin{align*}
    B_d &:= \sup_f |\preper(f,\QQ)|\in [0,\infty],\\
    C_d &:= \sup_{f,g}|\preper(f,\CC) \cap \preper(g,\CC)|\in [0,\infty],
\end{align*}
where the supremum defining $B_d$ is taken over all polynomials $f(x) \in \QQ[x]$ with $2 \leq \deg(f) \leq d$, and the supremum defining $C_d$ is taken over all $f(x), g(x) \in \CC[x]$ with $2 \leq \deg(f), \deg(g) \leq d$ such that $\preper(f,\CC) \neq \preper(g,\CC)$. Both $B_d$ and $C_d$ are conjectured to be finite for all $d\geq 2$.

Northcott \cite{northcott} proved that if $\deg(f)\geq 2$, then $\preper(f,\QQ)$ is finite.
The Morton-Silverman Uniform Boundedness Conjecture \cite[p. 100]{MS} asserts, in part, that $B_d < \infty$.  
This conjecture has motivated a substantial volume of work in arithmetic dynamics (see Silverman \cite[Sec. 3.3]{ADS}). While it is widely believed to be true, the Uniform Boundedness Conjecture has yet to be proved unconditionally in any degree. 
Looper \cite{looper} recently gave a conditional proof that $B_d < \infty$, assuming a generalization of the $abc$ conjecture. 

DeMarco, Krieger, and Ye \cite[Conj. 1.4]{DKY} conjecture that $C_d < \infty$ for all $d \geq 2$; they prove this conjecture when $f$ and $g$ are restricted to the family of quadratic polynomials of the form $x^2 + c$ \cite[Thm. 1.1]{DKY}. 
Mavraki and Schmidt \cite{MSch} recently proved an analogous uniform bound on the number of common preperiodic points along 1-parameter families in $\mathrm{Rat}_d \times \mathrm{Rat}_d$, where $\mathrm{Rat}_d$ denotes the space of degree-$d$ rational functions.

The polynomials asserted to exist in Theorem \ref{thm main intro} combined with an explicit family described below in Theorem \ref{thm d+6 intro} lead to the following lower bounds on $B_d$ and $C_d$.

\begin{cor}
\label{cor lower bounds}
For all integers $d\geq 2$,
\begin{enumerate}
    \item $B_d \geq d + \max(6,\lfloor \log_2(d)\rfloor)$,
    \item $C_d \geq d^2 + 4d + 1$.
\end{enumerate}
Furthermore, there are infinitely many $d\geq 2$ for which
\[
    C_d \ge d^2 + d\lfloor \log_2(d/4)\rfloor + 1.
\]
\end{cor}

One may compare Corollary~\ref{cor lower bounds}(1) to known lower bounds on $A_g := \sup_X |X(\QQ)|$, where $X$ ranges over all smooth irreducible genus-$g$ curves defined over $\QQ$. In this setting, the best known lower bound for $A_g$ that holds for all $g \ge 2$ is linear in $g$ (see \cite{CHM}), though it is unknown whether the correct upper bound should also be linear. In that spirit, we pose the following question:

\begin{question}
\label{quest lower bound}
What is the order of growth of $B_d$ as $d \to \infty$?
Is it true that 
\[
    B_d = d + O(\log(d))?
\]
\end{question}

\begin{remark}
Our proof of Theorem~\ref{thm main intro}(3) (hence also Corollary~\ref{cor lower bounds}(2)) actually shows something stronger: Given a set $\mathcal P$ of polynomials, we say that a finite set $S \subseteq \CC$ has a \emph{finite orbit} under $\mathcal P$ if $f(S) \subseteq S$ for every $f\in \mathcal P$. (See \cite{REU} for a detailed study of finite orbits for pairs of quadratic and cubic polynomials.) Note that if $S$ has a finite orbit under $\mathcal P$, then $S \subseteq \bigcap_{f \in \mathcal P} \preper(f,\CC)$, but, in general, common preperiodic points of the elements of $\mathcal{P}$ need not have a finite orbit under $\mathcal{P}$.

With this setup, we prove that for the polynomials $f_{d}(x) \in \QQ[x]$ provided by Theorem~\ref{thm main intro}, the set of maps $\mathcal P := \{f_{d}(x) + i : 0 \le i \le \lfloor\log_2(d)\rfloor\}$ has a finite orbit with at least $\deg(f_{d})(d - 1) + 1$ elements.
As a result, it follows that Corollary \ref{cor lower bounds}(2) holds when $C_d$ is replaced with
\[
    \widetilde{C}_d := \sup_{f,g} \sup_S |S|
    \le C_d,
\]
where $f$ and $g$ range over all polynomials of degree $2 \le \deg(f),\ \deg(g) \le d$ such that $\preper(f,\CC) \ne \preper(g,\CC)$ and $S$ ranges over all finite orbits of $\mathcal{P} = \{f,g\}$.
\end{remark}

\begin{remark}
The two uniform boundedness conjectures stated above for polynomials are believed to hold more generally for rational functions on $\PP^1$.
However, the methods of this paper appear to be constrained to polynomials.
Note that if $B_d'$ and $C_d'$ are 
defined analogously to $B_d$ and $C_d$, but with rational functions instead of polynomials, then we have $B_d < B_d'$ and $C_d < C_d'$ 
for all $d \ge 2$. This inequality follows from the simple observation that every polynomial is a rational function, plus the fact that we are not counting $\infty$ as a preperiodic point, though it is a fixed point for every polynomial map when considered as an endomorphism of $\PP^1$.
\end{remark}

\begin{remark}
Fu and Stoll \cite{FT} recently proved a result analogous to Corollary \ref{cor lower bounds}(2), giving lower bounds on the maximal number of common torsion $x$-coordinates for pairs of elliptic curves $E_1, E_2$ such that $x(E_{1,\mathrm{tors}}) \neq x(E_{2,\mathrm{tors}})$.
Their results have the following dynamical interpretation: If $f_i(x)$ denotes the degree-4 flexible Latt\`es map associated to multiplication by $2$ on the elliptic curve $E_i$, then $x(E_{i,\mathrm{tors}}) = \preper(f_i(x),\CC)$. Hence \cite[Thm. 2]{FT} implies that there are infinitely many pairs of elliptic curves $E_1$, $E_2$ for which 
\[
    22 \leq |\preper(f_1(x),\CC) \cap \preper(f_2(x),\CC))| < \infty,
\]
and \cite[Thm. 3]{FT} implies that there exists an explicit pair of elliptic curves $E_1$, $E_2$ such that 
\begin{equation}
\label{eqn fu_stoll}
    |\preper(f_1(x),\CC) \cap \preper(f_2(x),\CC))| = 34.
\end{equation}
Using the notation of the previous remark, \eqref{eqn fu_stoll} implies that $C_4' \geq 34$. On the other hand, in Section \ref{sec low degree} we provide an example that shows that $C_4 \geq 36$, hence that $C_4' \geq 37$.
\end{remark}

\begin{remark}
Despite the considerable interest in proving $B_d < \infty$, we are unaware of any previous work explicitly proving nontrivial lower bounds on $B_d$ outside of finitely many low degree cases. However, motivated by problems in complexity theory, Cohen, Shpilka, and Tal prove a result (\cite[Thm. 1.5]{cohen}; see also \cite[p. 458]{cohen}) that implies the following: For all $0 < \ve < 1$, there exists $d_\ve$ such that for all $d \ge d_\ve$, we have
\[
    B_d \geq d + \lfloor \ve \log_2(d) \rfloor.
\]
Our improvement on this lower bound in Corollary \ref{cor lower bounds}(1) stems from an exact evaluation of a certain lattice discriminant (see Theorem \ref{thm lattice disc}) that was only bounded in \cite{cohen}.
We thank Yan Sheng Ang for bringing \cite{cohen} to our attention.
\end{remark}

\subsection{Dynamical compression}

For a positive integer $m$, let $[m] := \{1,2,3,\ldots, m\}$.
We say a degree-$d$ polynomial $g(x) \in \CC[x]$ exhibits \emph{dynamical compression} if $g(x)$ is conjugate to some polynomial $f(x)$---that is, $f = \ell \circ g \circ \ell^{-1}$ for some linear polynomial $\ell(x) \in \CC[x]$---which
satisfies
\begin{equation*}
    f([m]) \subseteq [n]
\end{equation*}
for some $m \geq n > d + 1$.
In this case, $[m] \subseteq \preper(f(x),\QQ)$.
The polynomials $f_d(x)$ asserted to exist in Theorem \ref{thm main intro} all exhibit dynamical compression.

\begin{example}
\label{ex quadratic}
Let $f(x) := \frac{x^2 - 9x + 22}{2}$. One may check that
\begin{equation}
\label{eqn quadratic compression}
    f([8]) \subseteq [7].
\end{equation}
Therefore, both $f(x)$ and $f(x) + 1$ exhibit dynamical compression and have at least 8 rational preperiodic points, namely the elements of $[8]$.
In fact, it may be shown that $f(x)$ and $f(x) + 1$ have exactly 8 rational preperiodic points.
Hence $B_2 \geq 8$; in fact,
Poonen \cite{poonen} has conjectured that $B_2 = 8$.
The polynomials $f(x)$ and $f(x) + 1$ are simultaneously conjugate to $x^2 - \frac{29}{16}$ and $x^2 - \frac{21}{16}$, respectively.
These quadratic polynomials appear several times in the literature for their exceptional properties, including in DeMarco, Krieger, and Ye \cite{DKY}, Doyle, Faber and Krumm \cite{DFK}, Hindes \cite{hindes}, Morton and Raianu \cite{MR}, and Poonen \cite{poonen}.

\end{example}

\begin{example}
The cubic polynomial $f(x) := \frac{x^3 - 18x^2 + 89x - 66}{6}$ satisfies $f([11]) \subseteq [11]$.
Thus both $f(x)$ and $12 - f(x)$ exhibit dynamical compression and have at least 11 rational preperiodic points.
These examples share the current record with 8 other cubics, found by computational search in Benedetto et al. \cite[Table 2]{benedetto}, for the cubic polynomial with the most rational preperiodic points.
Of the 10 record-holding cubics found by Benedetto et al., only the two conjugate to $f(x)$ and $12 - f(x)$ exhibit dynamical compression.
\end{example}

The following proposition, proved in Section \ref{sec common}, makes explicit the connection between dynamical compression and polynomials with many common preperiodic points.

\begin{prop}
\label{prop common preper intro}
Suppose that $f(x) \in \CC[x]$ is a degree $d\geq 2$ polynomial such that
\[
    f([m]) \subseteq [n],
\]
for some integers $m > n \geq 1$. Then
\[
    d(n - 1) + 1 \leq \Big|\bigcap_{i = 0}^{m-n} \preper(f(x) + i,\CC)\Big| < \infty.
\]
\end{prop}

\begin{example}
\label{ex quadratic 2}
Returning to Example \ref{ex quadratic}, let $f(x) := \frac{x^2 - 9x + 22}{2}$. Then \eqref{eqn quadratic compression} and Proposition \ref{prop common preper intro} imply that 
\[
    |\preper(f(x),\CC) \cap \preper(f(x) + 1,\CC)| \geq 13.
\]
However, by comparing the preperiodic points with small forward orbit for $f(x)$ and $f(x) + 1$ directly, we find that $\preper(f(x),\CC) \cap \preper(f(x) + 1,\CC)$ actually contains at least 26 points.
Hence $C_2 \geq 26$.
That is, dynamical compression accounts for half of the known preperiodic points shared by $f(x)$ and $f(x) + 1$.
To the best of our knowledge, this is the current record for a lower bound on $C_2$. See Table \ref{table cd bounds} in Section \ref{sec low degree} for more lower bound records on $C_d$ for $2\leq d \leq 15$.
\end{example}

The proof of Theorem \ref{thm main intro} uses a geometry of numbers approach to show the existence of polynomials which compress exceptionally large intervals of integers but does not produce explicit examples.
Thus it remains an interesting problem to construct explicit polynomials which surpass the trivial lower bounds on $B_d$ and $C_d$. 
Our last result provides one such family $r_d(x)$.
Formulas for $r_d(x)$ are given in Section \ref{sec d + 6}.

\begin{thm}
\label{thm d+6 intro}
For all $d\geq 2$, there is an explicit degree-$d$ polynomial $r_d(x)$ such that
\[
    r_d([d+6]) \subseteq \begin{cases} [d+5] & \text{if $d$ is even,}\\ [d+4] & \text{if $d$ is odd.}\end{cases}
\]
\end{thm}

\subsection{Acknowledgements}
We thank Laura DeMarco for a helpful correspondence which prompted this work.
We thank Hang Fu for pointing out that \eqref{eqn fu_stoll} follows from \cite[Thm. 3]{FT}. We thank Joseph Silverman for several helpful comments on an earlier draft, as well as for suggesting that we include the discussion and question immediately following Corollary~\ref{cor lower bounds}.
We thank Yan Sheng Ang for bringing the article \cite{cohen} to our attention.
And we thank the University of Chicago Research Computing Center for access to their high performance computing cluster.
John Doyle was partially supported by NSF grant DMS-2112697.
Trevor Hyde was partially supported by the NSF Postdoctoral Research Fellowship DMS-2002176 and the Jump Trading Mathlab Research Fund.

\section{Rational Preperiodic Points}
\label{sec rational preper}

The goal of this section is to prove Theorem \ref{thm main intro}, which is stated in a refined form below as Theorem \ref{thm main}.
Our strategy is to consider the lattice $\Lambda_{d,e}$ generated by vectors of the form $(f(0), f(1), \ldots, f(d+e))$ where $f(x)$ is a degree-at-most-$d$ integer-valued polynomial and $e > 0$ is an integer. (Fecall that $g(x) \in \QQ[x]$ is said to be \emph{integer-valued} if $g(\ZZ) \subseteq \ZZ$.)
There is a natural bijection between vectors in $\Lambda_{d,e}$ contained within a small box near the origin and degree-at-most-$d$ polynomials exhibiting dynamical compression.
We use a classical geometry-of-numbers theorem of Minkowski to prove the existence of lattice points in this box by analyzing the discriminant of $\Lambda_{d,e}$.

\subsection{Lattices and their discriminants}
\label{sec lattices}
Let $m\geq 1$ be an integer. By a \emph{lattice} $\Lambda \subseteq \RR^m$ we mean a discrete free abelian subgroup of $\RR^m$. 
If $\Lambda$ is a rank $n$ lattice with basis $v_1, v_2, \ldots, v_n$, then we call the compact set $\{\sum_{i=1}^n c_i v_i : 0 \leq c_i \leq 1\}$ a \emph{fundamental domain} of $\Lambda$. 
The \emph{discriminant} of $\Lambda$, which we denote by $\delta(\Lambda)$ is the square of the $n$-dimensional volume of a fundamental domain of $\Lambda$. If $M$ is the matrix with rows $v_i$, then 
\[
    \delta(\Lambda) = \det(MM^T).
\]
Note that $\delta(\Lambda)$ is  independent of the choice of basis for $\Lambda$.

Given integers $d, e \geq 0$, let $\Lambda_{d,e}$ be the lattice in $\RR^{d+e+1}$ spanned by the $d+1$ vectors
\[
    u_i := \left(\binom{0}{i}, \binom{1}{i},\ldots, \binom{d+e}{i}\right) \in \ZZ^{d+e+1}
\]
for $0 \leq i \leq d$. (Note that $\binom{j}{i} = 0$ if $j < i$.)
The lattice $\Lambda_{d,e}$ has rank $d + 1$: Indeed, if $\sum_{i=0}^d a_iu_i = 0$, then the degree-at-most-$d$ polynomial $\sum_{i=0}^d a_i\binom{x}{i}$ vanishes at more than $d$ points (namely, the points $0,1,2,\ldots,d+e$), hence all of the coefficients must be zero.
We now provide an explicit formula for $\delta(\Lambda_{d,e})$.

\begin{thm}
\label{thm lattice disc}
The discriminant of $\Lambda_{d,e}$ is given by
\[
    \delta(\Lambda_{d,e})
    = \prod_{i=0}^d\prod_{j=1}^e \frac{d + i + j + 1}{i + j}.
\]
\end{thm}

\begin{proof}
Let $M_{d,e} = \big(\binom{j}{i}\big)$ be the $(d+1)\times(d+e+1)$ matrix with rows $u_i$. Thus,
\[
    \delta(\Lambda_{d,e}) = \det(M_{d,e} M_{d,e}^T).
\]
To evaluate $\delta(\Lambda_{d,e})$ we use the Lindstr\"om-Gessel-Viennot lemma \cite[Thm. 1]{GV} 
to interpret $\det(M_{d,e} M_{d,e}^T)$ as the number of plane partitions which fit inside a $(d+1)\times (d+1)\times e$ box. 
The number of such plane partitions is given by MacMahon's formula, which provides the desired product formula for $\delta(\Lambda_{d,e})$.

A \emph{plane partition} $\Pi$ is a finite, weakly increasing sequence of partitions $\lambda_1 \leq \lambda_2 \leq \cdots \leq \lambda_k$. 
More intuitively, a plane partition may be thought of us a finite set of boxes stacked in the corner of a room.
For example, the plane partition in Figure \ref{fig:plane partition tower} may be visualized as the stack of boxes in Figure \ref{fig:plane partition boxes}.

Given positive integers $r, s, t$, we say a plane partition $\Pi: \lambda_1 \leq \lambda_2 \leq \cdots \leq \lambda_k$ \emph{fits inside the $r\times s \times t$ box} if $k\leq t$ and if the Young diagram of $\lambda_k$ fits inside an $r\times s$ box. Equivalently, the box diagram of $\Pi$ fits inside an $r\times s \times t$ box. Let $N(r,s,t)$ denote the number of plane partitions which fit inside an $r\times s \times t$ box.
We claim that $\delta(\Lambda_{d,e}) = N(d+1,d+1,e)$.

Let $\I$ be the $\ZZ$-lattice in $\RR^2$ spanned by the vectors $v_1 := (-\sqrt{3},1)$ and $v_2 := (\sqrt{3},1)$. Let $v_3 := v_1 + v_2 = (0,2) \in \I$. Given positive integers $a, b, c$, let $H(a,b,c)$ denote the convex hull of the six points 
\[
    \{0, av_1, bv_2, av_1 + cv_3, bv_2 + cv_3, av_1 + bv_2 + cv_3\} \subseteq \I.
\]
See Figure \ref{fig:H} for an illustration.
There is a simple and well-known correspondence between plane partitions that fit inside a box of dimensions $a\times b\times c$ and rhombic tilings of $H(a,b,c)$ (see Figure \ref{fig:rhombic tiling}).
This correspondence comes from the interpretation of a plane partition as a stack of cubical blocks in a the first quadrant of $\RR^3$ and viewing this stack of blocks from along the ray spanned by $(1,1,1)$.

\begin{figure}[H]
    \centering
    \includegraphics[scale=.13]{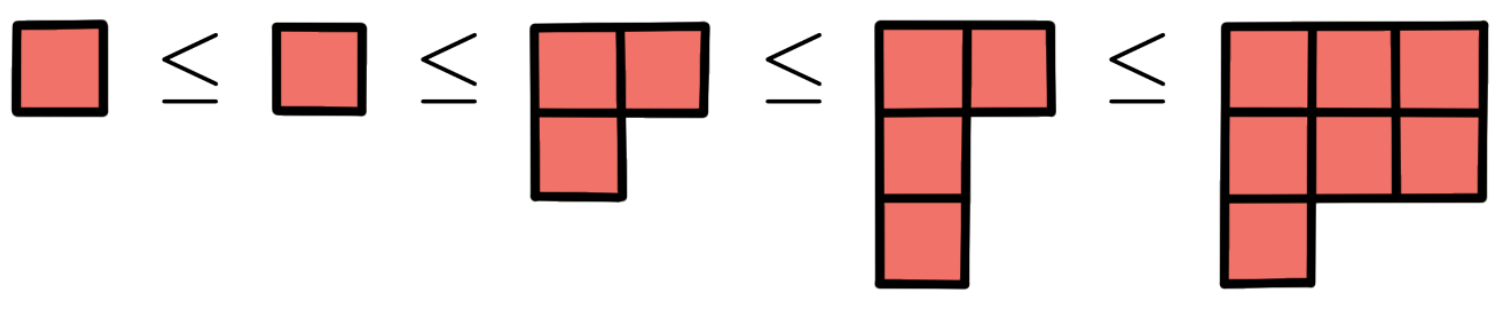}
    \caption{Example of a plane partition, illustrated using Young diagrams.}
    \label{fig:plane partition tower}
\end{figure}
\begin{figure}
\centering
\begin{minipage}{.47\textwidth}
  \centering
  \includegraphics[scale=.13]{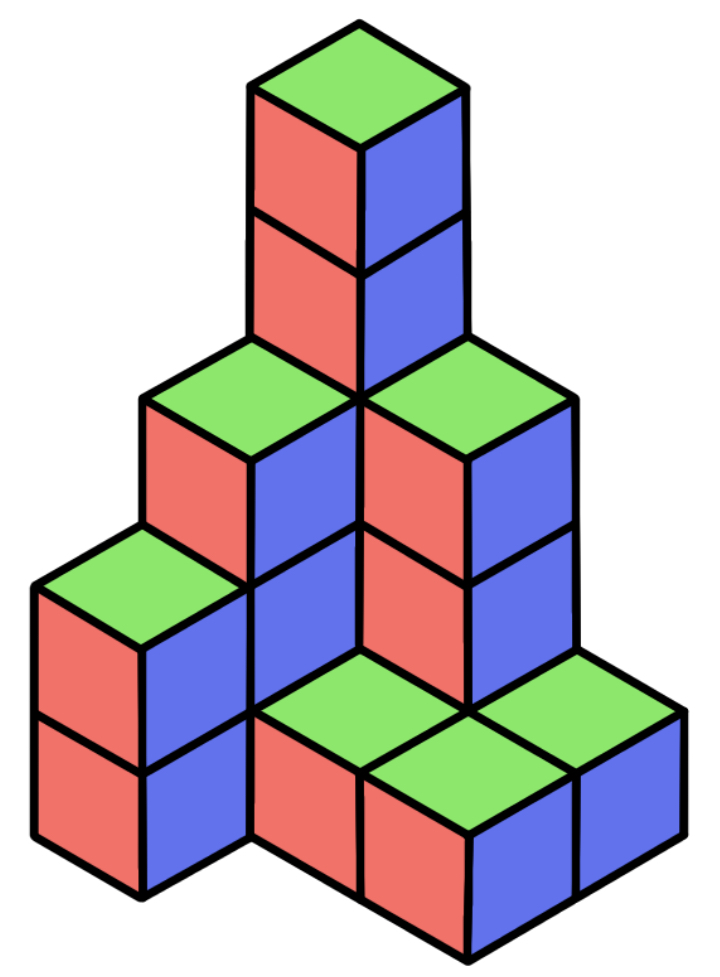}
  \captionof{figure}{Plane partition as stack of boxes}
  \label{fig:plane partition boxes}
\end{minipage}
\begin{minipage}{.47\textwidth}
  \centering
  \includegraphics[scale=.2]{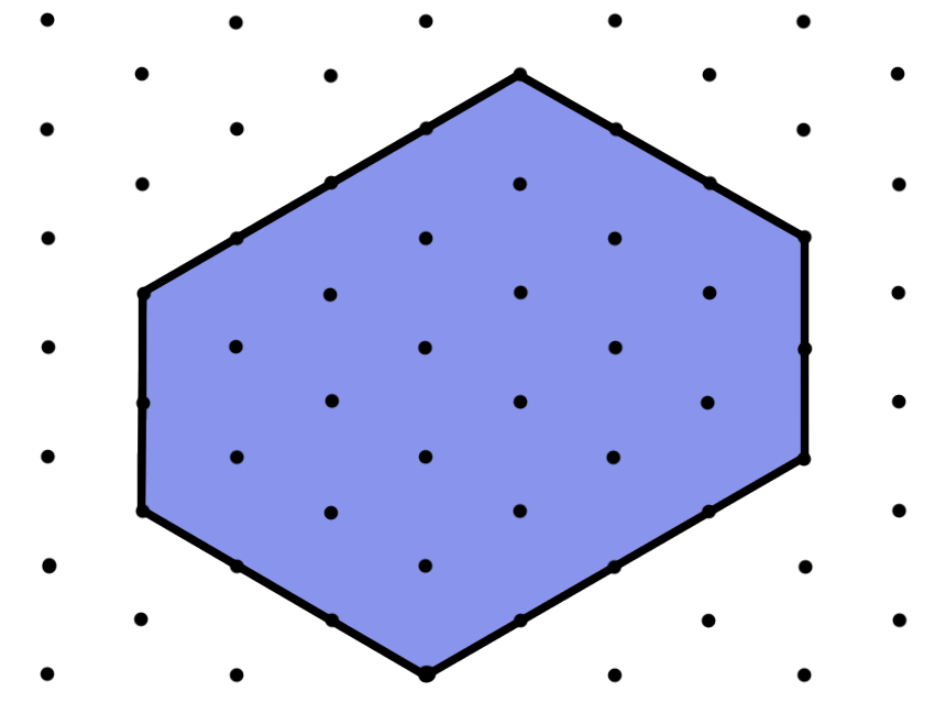}
  \captionof{figure}{$H(3,4,2)$ and a portion of the lattice $\I$.}
  \label{fig:H}
\end{minipage}
\end{figure}

Now consider a plane partition $\Pi$ contained in a box of size $(d+1)\times(d+1)\times e$ viewed as a rhombic tiling of $H(d+1,d+1,e)$.
There are three types of rhombic tiles $T_i$, characterized by which $v_i$ is parallel to the short diagonal of $T_i$; see Figure \ref{fig:tile types}. 

\begin{figure}[H]
\centering
\begin{minipage}{.47\textwidth}
  \centering
  \includegraphics[scale=.13]{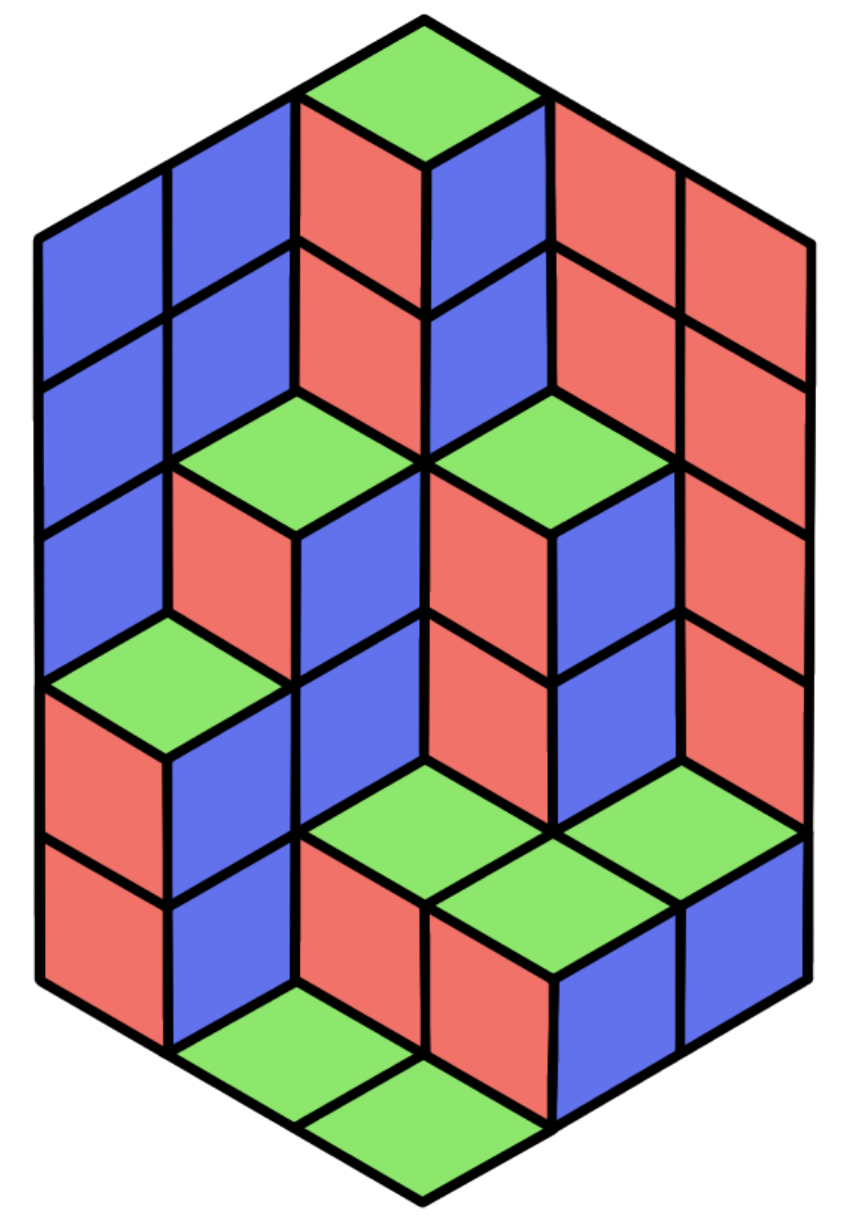}
  \captionof{figure}{Rhombic tiling of $H(3,3,5)$ corresponding to the plane partition from Figure \ref{fig:plane partition boxes}.}
  \label{fig:rhombic tiling}
\end{minipage}
\begin{minipage}{.47\textwidth}
  \centering
  \includegraphics[scale=.105]{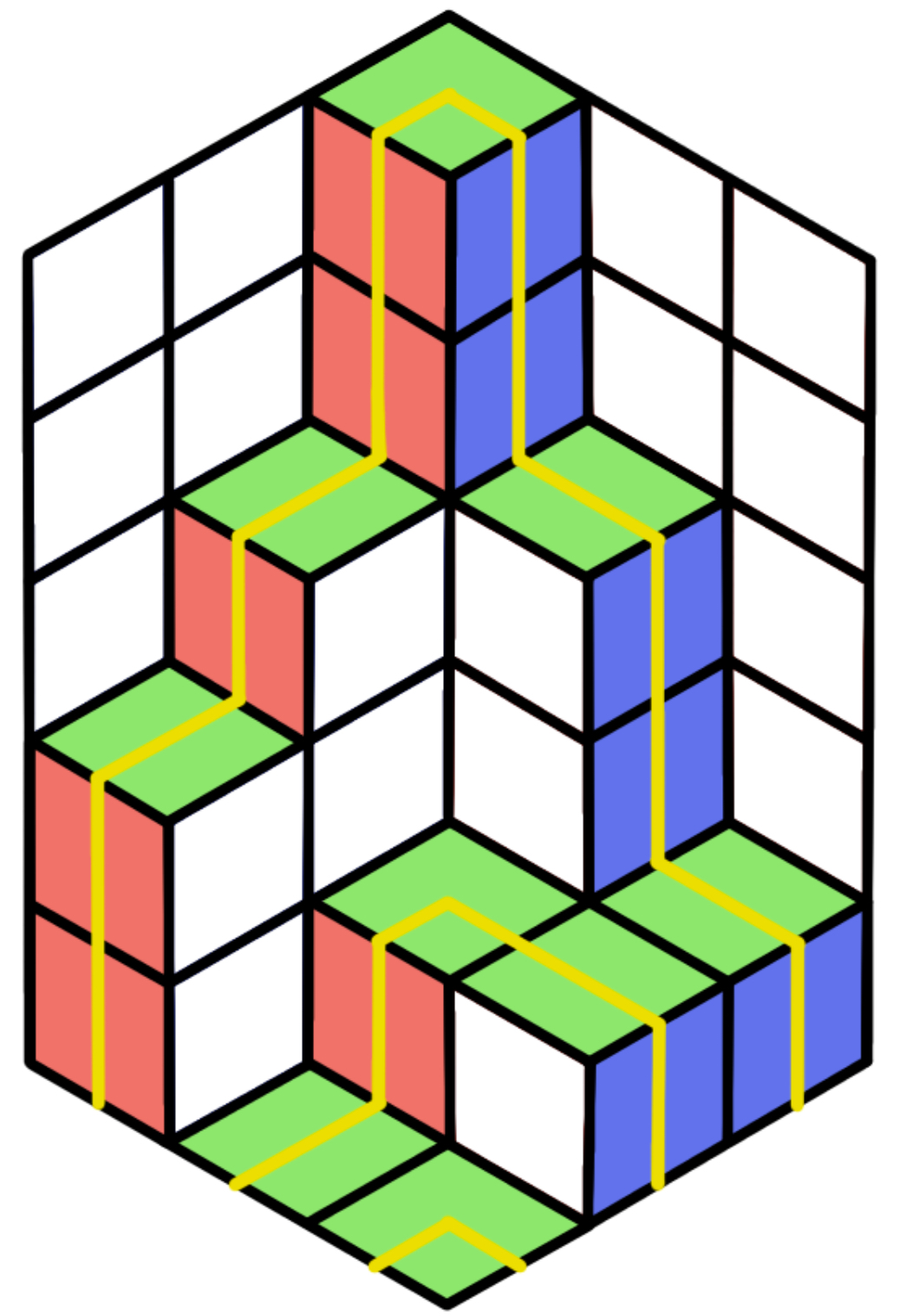}
  \captionof{figure}{Plane partition in $3\times 3\times 5$ box with corresponding rhombic paths highlighted.}
  \label{fig:plane partition to paths}
\end{minipage}
\end{figure}

For each $0\leq i \leq d$, let $e_i$ denote the line segment from $iv_1$ to $(i+1)v_1$ and let $f_i$ denote the line segment from $iv_2$ to $(i+1)v_2$. 
If $0\leq i, j \leq d$, then we define a \emph{rhombic path} from $e_i$ to $f_j$ to be a sequence of rhombic tiles starting at $e_i$, ending at $f_j$, such that to the left of the $v_3$-axis, each of the tiles is of type 
$T_2$ or $T_3$, and to the right of the $v_3$-axis each of the tiles is of type $T_1$ or $T_3$ (see Figure \ref{fig:rhombic path}).
Equivalently, to each $(d+1)\times (d+1)\times e$ box we may associate an acyclic directed graph $G_{d,e}$ with disjoint sets of vertices labelled $e_i$ and $f_j$ such that rhombic paths from $e_i$ to $f_j$ correspond to directed paths from $e_i$ to $f_j$ in $G_{d,e}$ (see Figure \ref{fig:digraph}.)

A plane partition $\Pi$ determines a sequence of non-crossing rhombic paths $P_i(\Pi)$ from $e_i$ to $f_i$.
Figure \ref{fig:plane partition to paths} illustrates the collection of rhombic paths associated to a plane partition contained in the $3\times 3 \times 5$ box.
Conversely, any sequence $P_0, P_1,\dots, P_d$ of non-intersecting rhombic paths $P_i$ from $e_i$ to $f_i$ determines a unique plane partition contained in a $(d+1)\times(d+1)\times e$ box.

\begin{figure}[H]
    \centering
    \includegraphics[scale=.12]{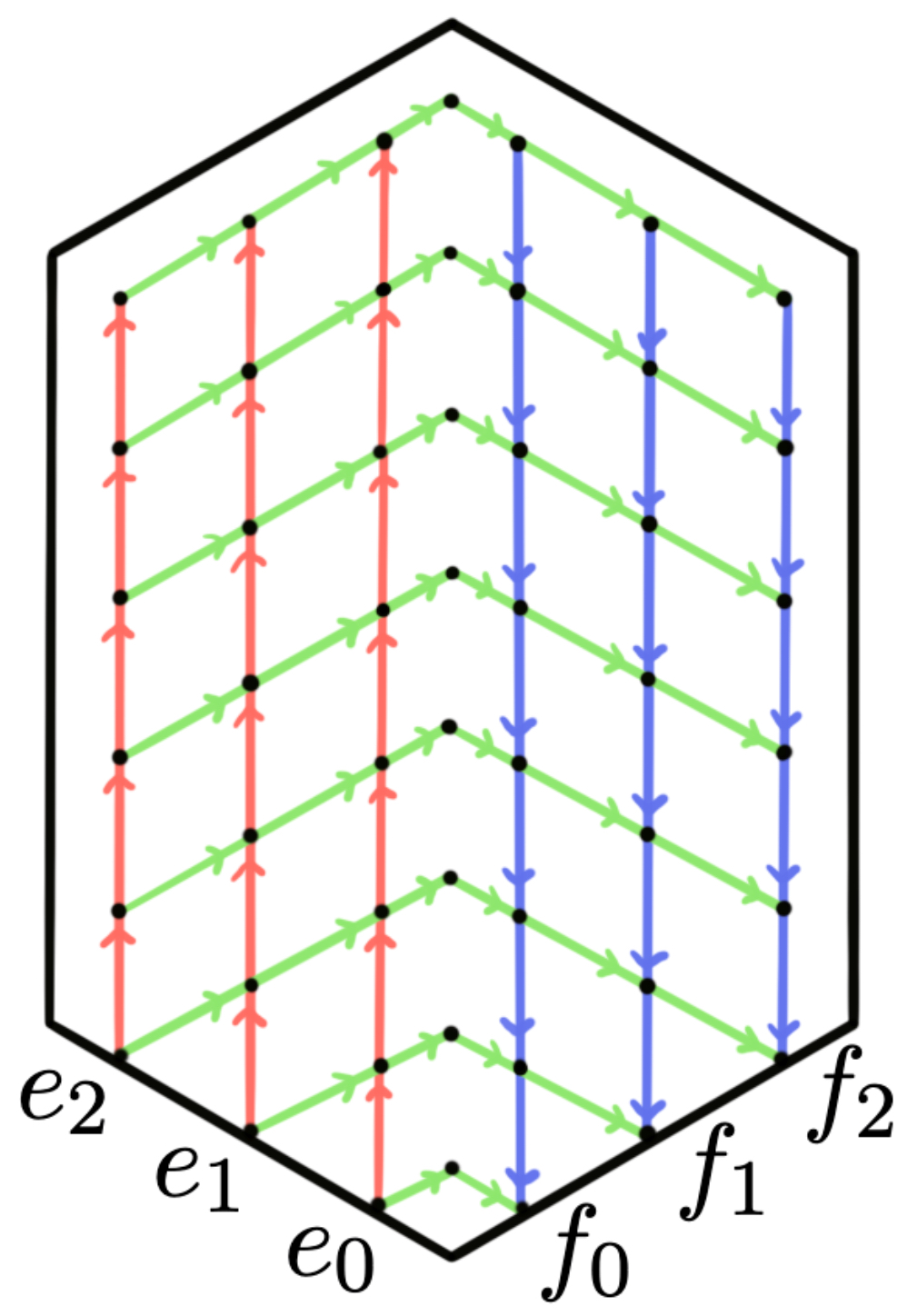}
    \caption{The acyclic digraph $G_{2,5}$.}
    \label{fig:digraph}
\end{figure}

Note that any path $P_i$ crosses the $v_3$-axis at a unique type $T_3$ tile $R_k$ 
with lowest point $kv_3$ for some $0\leq k \leq d + e$.
Every rhombic path from $e_i$ to $R_k$ consisting only of tiles of type $T_2$ and $T_3$ has length $k$ and contains exactly $i$ tiles of type $T_3$.
Hence there are $\binom{k}{i}$ such rhombic paths.
Therefore, the $ik$th entry of $M_{d,e}$ counts the number of rhombic paths from $e_i$ to $R_k$, and by symmetry it follows that the $ij$th entry of $M_{d,e}M_{d,e}^T$ counts the number of rhombic paths from $e_i$ to $f_j$.
Therefore the Lindstr\"om-Gessel-Viennot lemma \cite[Cor. 2]{GV} applied to the acyclic digraph $G_{d,e}$ implies that $\delta(\Lambda_{d,e}) = \det(M_{d,e}M_{d,e}^T) = N(d+1,d+1,e)$ is the total number of non-crossing rhombic paths, hence the total number of plane partitions which fit inside a box of dimension $(d+1)\times(d+1)\times e$. 
On the other hand, MacMahon's theorem \cite[p. 378]{stanley} implies that
\[
    N(d+1,d+1,e) = \prod_{i=0}^d\prod_{j=1}^e \frac{d + i + j + 1}{i + j}.\qedhere
\]

\begin{figure}[H]
\centering
\begin{minipage}{.47\textwidth}
  \centering
  \includegraphics[scale=.15]{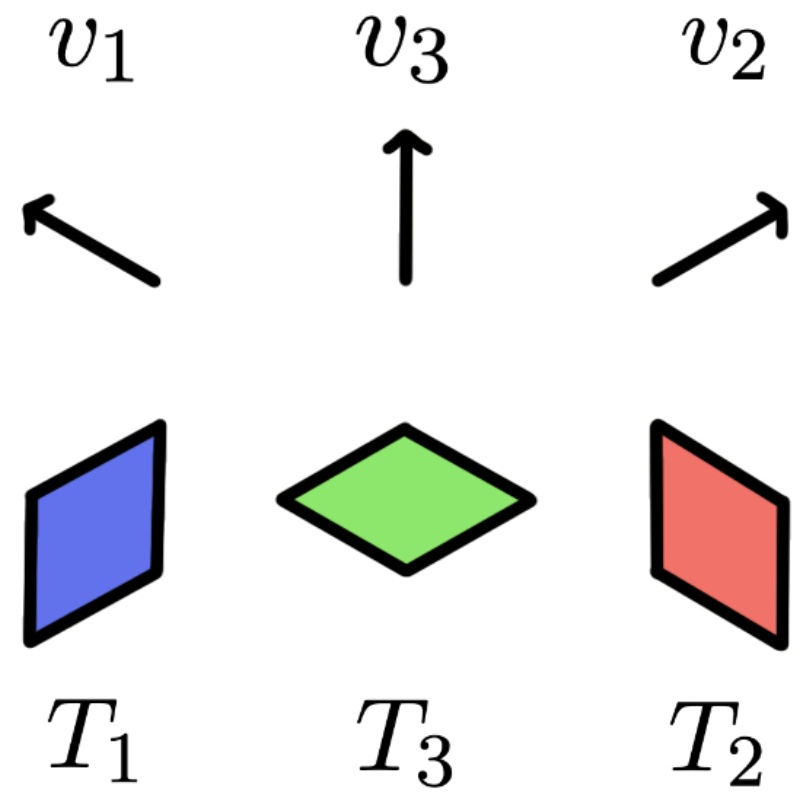}
  \captionof{figure}{Three rhombic tiles.}
  \label{fig:tile types}
\end{minipage}
\begin{minipage}{.47\textwidth}
  \centering
  \includegraphics[scale=.15]{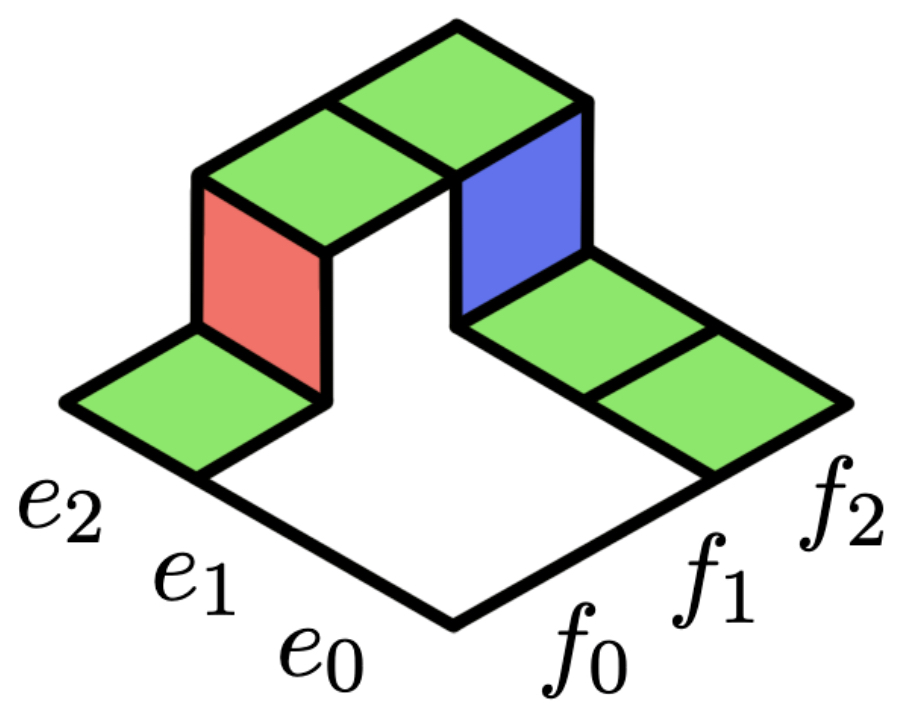}
  \captionof{figure}{Example of rhombic path from $e_2$ to $f_2$.}
  \label{fig:rhombic path}
\end{minipage}
\end{figure}
\end{proof}

\begin{remark}
The use of the Lindstr\"om-Gessel-Viennot lemma to count plane partitions and to evaluate determinants of matrices with binomial coefficient entries is not new; however, we did not find the evaluation of $\delta(\Lambda_{d,e})$ among the known results. 
Many variations on this idea may be found in the literature (see, for example, \cite{GV2, GV, stembridge}).
\end{remark}

The following corollary extracts an upper bound on $\delta(\Lambda_{d,e})$ from the product formula in Theorem \ref{thm lattice disc} that will be used in the proof of Theorem \ref{thm main}.

\begin{cor}
\label{cor disc bound}
Let $d, e$ be integers such that $d \ge 33$ and $1 \le e \le \log_2(d/2)$. Then
	\begin{equation}\label{eq:disc_bound}
		\log_2(\delta(\Lambda_{d,e})) < 2(d - 1)\log_2(d/2) - e\log_2(d + e + 1).
	\end{equation}
\end{cor}

\begin{proof}
Theorem~\ref{thm lattice disc} shows that $\delta(\Lambda_{d,e})$
is an increasing function of $e$.
Thus, it suffices to fix $e := \lfloor \log_2(d/2)\rfloor$ and show that
\begin{equation}\label{eq:new_goal}
    \log_2(\delta(\Lambda_{d,e})) < 2(d-1)e - e\log_2(d + e + 1).
\end{equation}
Taking the logarithm of the product formula
\[
 	\delta(\Lambda_{d,e}) = \prod_{i=0}^d\prod_{j=1}^e \frac{d+1+i+j}{i+j}
\]
yields
	\begin{align}
		\begin{split}
		\log_2(\delta(\Lambda_{d,e}))
			&= \sum_{i=0}^d\sum_{j=1}^e \log_2\left(\frac{d+1+i+j}{i+j}\right)
		\end{split}\notag\\
		\begin{split}
			&= \sum_{n=1}^{d+e} \#\{(i,j) : 0 \le i \le d,\ 1 \le j \le e,\ i + j = n\} \cdot \log_2\left(\frac{d+1+n}{n}\right)
		\end{split}\notag\\
		\begin{split}
			&=  
				\sum_{n=1}^e n\log_2\left(1 + \frac{d+1}{n}\right)
				+ \sum_{n=d+1}^{d+e} (d+e+1-n)\log_2\left(1 + \frac{d+1}{n}\right)\\
				&\qquad	+ e\sum_{n=e+1}^d \log_2\left(1 + \frac{d+1}{n}\right).
				\end{split}
			\label{eq:separate_sums}
	\end{align}
Note that $\log_2\left(1 + \frac{d+1}{n}\right)$ is a positive, decreasing function of $n$. Hence,
\[
    \sum_{n=1}^e n\log_2\left(1 + \frac{d+1}{n}\right) + \sum_{n=d+1}^{d+e} (d+e+1-n)\log_2\left(1 + \frac{d+1}{n}\right)
    \leq \binom{e+1}{2}(\log_2(d+2)+1).
\]
Note that for any $\ve > 0$ and for all sufficiently large $d$ (depending on $\ve$),
\begin{align*}
    \binom{e+1}{2}(\log_2(d+2) + 1) &= \binom{e+1}{2}\log_2(d) + \binom{e+1}{2}\log_2\Big(1 + \frac{2}{d}\Big) + \binom{e+1}{2}\\
    &< \Big(\frac{1}{2} + \ve\Big)e^2\log_2(d) + e^2.
\end{align*}
In particular, for if we take $\ve = .03$, then for all $d\geq 256$ we have
\begin{equation}
\label{eqn easy bound}
    \sum_{n=1}^e n\log_2\left(1 + \frac{d+1}{n}\right) + \sum_{n=d+1}^{d+e} (d+e+1-n)\log_2\left(1 + \frac{d+1}{n}\right)
    < .53e^2\log_2(d) + e^2.
\end{equation}

Interpreting the third sum in \eqref{eq:separate_sums} as a right hand Riemann sum gives us
\begin{align*}
    \sum_{n=e+1}^d \log_2\left(1 + \frac{d+1}{n}\right) &\leq \int_e^{d}\log_2\Big(1 + \frac{d+1}{x}\Big)\,dx\\
    &= (2d+1)\log_2(2d+1) - d\log_2(d) - (d+e+1)\log_2(d+e+1) + e\log_2(e).
\end{align*}
Observe that for $d\geq 2$,
\begin{align*}
    (2d+1)\log_2(2d+1) &= 2d + 1 + (2d+1)\log_2(d) + (2d+1)\log_2\Big(1+ \frac{1}{2d}\Big)\\
    &\leq 2d + (2d+1)\log_2(d) + 3,
\end{align*}
and
\begin{align*}
    (d+e+1)\log_2(d+e+1) &= (d+e+1)\log_2(d) + (d+e+1)\log_2\Big(1 + \frac{e+1}{d}\Big)\\
    &\geq (d+e+1)\log_2(d) + e + 1.
\end{align*}
Hence for $d\geq 2$, 
\begin{align}
    \sum_{n=e+1}^d \log_2\left(1 + \frac{d+1}{n}\right) &
    \leq 2d + (2d+1)\log_2(d) + 3 - d\log_2(d) \notag\\
    &
    \qquad- (d+e+1)\log_2(d) - e - 1 + e\log_2(e)\notag\\
    &=2d - e\log_2(d) + e\log_2(e) - e + 2.\label{eqn harder bound}
\end{align}
Combining the estimates \eqref{eqn easy bound} and \eqref{eqn harder bound} for $d\geq 256$ we have
\begin{align*}
    \log_2(\delta(\Lambda_{d,e})) &< (2de - e^2\log_2(d) + e^2\log_2(e) - e^2 + 2e) + .53e^2\log_2(d) + e^2\\
    &= 2de - .47e^2\log_2(d) + e^2\log_2(e) + 2e.
\end{align*}
It therefore suffices to prove that
\[
    2de - .47e^2\log_2(d) + e^2\log_2(e) + 2e \leq 2(d-1)e - e\log_2(d+e+1),
\]
which is equivalent to
\begin{equation}
\label{eqn final reduction}
    e^2\log_2(e) + e\log_2(d+e+1) + 4e \leq .47e^2\log_2(d),
\end{equation}
which can be shown to hold for all $d \ge 1079$.
One may then check by computation that the inequality \eqref{eq:disc_bound} holds for $33 \le d \le 1079$, completing the proof.
\end{proof}

\begin{remark}
Cohen, Shpilka, and Tal \cite[Lem. 6.4]{cohen} proved, in our notation, that
\[
    \log_2(\delta(\Lambda_{d,e})) \leq (2d+e + 1)e.
\]
If $e\leq \log_2(d/2)$, this gives the bound
\begin{equation}
\label{eqn cst bound}
    \log_2(\delta(\Lambda_{d,e})) \leq 2d\log_2(d/2) + \log_2(d)^2 - \log_2(d).
\end{equation}
This bound has the same leading term as the bound we prove in Corollary \ref{cor disc bound}, but the lower order terms in \eqref{eqn cst bound} do not suffice to prove Theorem \ref{thm main}.
\end{remark}

\subsection{Minkowski's theorem on successive minima}
Suppose that $K \subseteq \RR^m$ is a compact, convex, centrally symmetric set. If $\Lambda \subseteq \RR^m$ is a rank-$m$ lattice, then for $1\leq i \leq m$ the \emph{$i$th successive minimum} of $\Lambda$ with respect to $K$, denoted $\lambda_i(\Lambda, K)$, is defined by
\[
    \lambda_i(\Lambda,K) := \min\{r \in \RR_{\geq 0} : \mathrm{span}(rK \cap \Lambda) \text{ has rank at least }i\}.
\]
Note that the $\lambda_i(\Lambda,K)$ are positive and weakly increasing with $i$.
The following classical theorem of Minkowski relates the successive minima, the volume of $K$, and the discriminant of $\Lambda$. See \cite[Chp. VIII, Thm. V]{cassels}.
\begin{thm}[Minkowski]
\label{thm mink}
Let $m\geq 1$, let $K \subseteq \RR^m$ be a compact, convex, centrally symmetric set, and let $\Lambda \subseteq \RR^m$ be a rank-$m$ lattice. Then
\[
    \vol(K)\prod_{i=1}^m \lambda_i(\Lambda,K) \leq 2^m \sqrt{\delta(\Lambda)}.
\]
\end{thm}

Suppose that $K = [-1,1]^m \subseteq \RR^m$.
Observe that, for $v \in \RR^m$ and $r \ge 0$, $v \in rK$ if and only if $||v||_\infty \leq r$ where 
\[
    ||(a_1, a_2, \ldots, a_m)||_\infty := \max_i |a_i|
\]
is the usual max norm on $\RR^m$.
We define $\lambda_i(\Lambda) := \lambda_i(\Lambda, [-1,1]^m)$.
Since $\vol([-1,1]^m) = 2^m$, we have the following useful direct corollary of Theorem \ref{thm mink}.

\begin{cor}
\label{cor simple mink}
If $\Lambda \subseteq \RR^m$ is a rank-$m$ lattice with discriminant $\delta(\Lambda)$, then
\[
    \prod_{i=1}^m \lambda_i(\Lambda) \leq \sqrt{\delta(\Lambda)}.
\]
\end{cor}

\subsection{Main Result}
We now prove the first main result. Recall that for $d \geq 2$,
\[
    B_d := \sup_f |\preper(f,\QQ)|\in [0,\infty],
\]
where the supremum is taken over all $f(x) \in \QQ[x]$ with $2 \leq \deg(f) \leq d$.

\begin{thm}
\label{thm main}
Let $d \geq 2$ be an integer.
Then there exists a polynomial $f_{d}(x) \in \QQ[x]$ with $2 \leq \deg(f_{d}) \leq d$ such that
\[
    f_{d}([d + \lfloor \log_2(d)\rfloor]) \subseteq [d].
\]
Hence for all $d\geq 2$,
\[
    B_d \geq d + \lfloor \log_2(d)\rfloor.
\]
\end{thm}

\begin{proof}
Let $d\geq 2$ and $e\geq 0$ be integers. Recall the rank-$(d+1)$ lattice $\Lambda_{d,e} \subseteq \RR^{d+e+1}$ constructed in Section \ref{sec lattices} spanned by the $d+1$ vectors
\[
    u_i := \left(\binom{0}{i}, \binom{1}{i},\ldots, \binom{d+e}{i}\right) \in \ZZ^{d+e+1}
\]
for $0 \leq i \leq d$.
Let $M$ be a real number such that $M > \frac{d\sqrt{d+e+1}}{2}$ and let $v_1, v_2, \ldots, v_{e}$ be an orthogonal basis for the orthogonal complement of $\Lambda_{d,e}$ in $\RR^{d+e+1}$ such that $||v_j|| = M$ for all $j$.
Define $\wt{\Lambda}_{d,e}$ to be the rank-$(d + e + 1)$ lattice spanned by the $u_i$ and $v_j$ with $0\leq i \leq d$ and $1 \leq j \leq e$.
Note that for any vector $w \in \wt{\Lambda}_{d,e}$ supported on some $v_j$ we have
\begin{equation}
\label{eqn support lower bound}
    ||w||_\infty \geq \frac{||w||}{\sqrt{d+e+1}} 
    \geq \frac{||v_j||}{\sqrt{d+e+1}}
    = \frac{M}{\sqrt{d+e+1}}
    > \frac{d}{2}.
\end{equation}
We claim that for $d\geq 33$ and $e := \lfloor \log_2(d/2)\rfloor = \lfloor\log_2(d)\rfloor - 1$ we have
\begin{equation}
\label{eqn lambda3 claim}
    \lambda_3(\wt{\Lambda}_{d,e}) < \frac{d}{2}.
\end{equation}
First we finish the proof of the theorem assuming \eqref{eqn lambda3 claim}, and then we return to prove the claim.

If \eqref{eqn lambda3 claim} holds, then there are three linearly independent vectors $w_1, w_2, w_3 \in \wt{\Lambda}_{d,e}$ such that $||w_i||_\infty < \frac{d}{2}$.
Thus \eqref{eqn support lower bound} implies that each $w_i$ is supported only on the $u_j$ with $0 \leq j \leq d$.
Linear independence implies that at least one $w_i$ must be supported on a $u_j$ with $j\geq 2$.
Suppose without loss of generality that $w_1 = \sum_{i=0}^{d}a_i u_i$ where $a_i \in \ZZ$ and $a_i \neq 0$ for some $i\geq 2$. If $g(x) := \sum_{i=0}^{d} a_i \binom{x}{i}$, then $w_1 = (g(0), g(1),\ldots, g(d+e))$. 
Hence $|g(i)| \leq ||w_1||_\infty < d/2$ for $0\leq i \leq d + e$.
Note that $g(x)$ is an integer-valued polynomial. 
Thus $f_{d}(x) := g(x-1) + \lfloor d/2\rfloor + 1$ is an integer-valued polynomial with $2 \leq \deg(f_{d}) \leq d$ such that for all $i \in [d+e+1] = [d + \lfloor \log_2(d)\rfloor]$,
\[
    0 \leq \lfloor d/2\rfloor - d/2 + 1 < f_{d}(i) < \lfloor d/2\rfloor + d/2 + 1 \leq d + 1.
\]
Therefore $f_{d}([d + \lfloor \log_2(d)\rfloor]) \subseteq [d]$, as we wished to show.

Now we turn to proving \eqref{eqn lambda3 claim}.
By construction, we have
\[
    \delta(\wt{\Lambda}_{d,e}) = \delta(\Lambda_{d,e})M^{2e}.
\]
If $i > d + 1$, then any set of $i$ independent vectors in $\wt{\Lambda}_{d,e}$ must contain at least one vector $w$ supported on some $v_j$.
Hence
$||w||_\infty \ge M/\sqrt{d+e+1}$
by \eqref{eqn support lower bound}.
Thus for $i > d + 1$,
\[
    \lambda_i(\wt{\Lambda}_{d,e}) \geq \frac{M}{\sqrt{d+e+1}} > \frac{d}{2}.
\]
The vectors $u_i$ all have integer entries, hence $\lambda_1(\wt{\Lambda}_{d,e}) \geq 1$.
The monotonicity of the $\lambda_i(\wt{\Lambda}_{d,e})$ gives us
\[
    \prod_{i=1}^{d+e+1}\lambda_i(\wt{\Lambda}_{d,e}) \geq \lambda_3(\wt{\Lambda}_{d,e})^{d-1}\Big(\frac{M}{\sqrt{d+e+1}}\Big)^{e}.
\]
Therefore, Corollary \ref{cor simple mink} implies that
\[
    \lambda_3(\wt{\Lambda}_{d,e})^{d-1}\Big(\frac{M}{\sqrt{d+e+1}}\Big)^{e} \leq \prod_{i=1}^{d+e+1}\lambda_i(\wt{\Lambda}_{d,e})
    \leq \sqrt{\delta(\wt{\Lambda}_{d,e})} 
    = \sqrt{\delta(\Lambda_{d,e})} \cdot M^{e},
\]
from which we conclude that
\[
    \log_2(\lambda_3(\wt{\Lambda}_{d,e})) \leq \frac{\log_2(\delta(\Lambda_{d,e})) + e\log_2(d+e+1)}{2(d-1)}.
\]
Corollary \ref{cor disc bound} implies that
\[
    \log_2(\delta(\Lambda_{d,e})) < 2(d-1)\log_2(d/2) - e\log_2(d+e+1)
\]
for $d\geq 33$, hence
\[
    \log_2(\lambda_3(\wt{\Lambda}_{d,e})) < \log_2(d/2),
\]
which is equivalent to \eqref{eqn lambda3 claim}. 

Now suppose that $2 \leq d \leq 32$. If $2 \leq d \leq 7$, then $\lfloor\log_2(d/2)\rfloor \leq 1$. 
Lagrange interpolation immediately implies the existence of polynomials $f_d(x) \in \QQ[x]$ with degree $d$ such that $f_d([d+1]) \subseteq [d]$. 
If $8 \leq d \leq 32$, then $\log_2(d/2) \leq 4$ and the sequence of polynomials $t_d(x)$ constructed in Corollary \ref{cor edge case family} satisfies
\[
    t_d([d+4]) \subseteq [5] \subseteq [d],
\]
which suffices to complete the proof.
\end{proof}

\begin{remark}
The idea of augmenting the lattice $\Lambda_{d,e}$ by arbitrary long vectors is borrowed from the proof of \cite[Theorem 1.5]{cohen}. This strategy greatly simplifies our original approach.
\end{remark}

\section{Common preperiodic points}
\label{sec common}

In this section we prove Proposition \ref{prop common preper intro} and part of Corollary \ref{cor lower bounds}(2), restated as Proposition \ref{prop common preper} and Theorem \ref{thm Cd lower bound} below.

\begin{lemma}
\label{lem: preimage}
If $f(x) \in \CC[x]$ is a degree-$d$ polynomial and $S \subseteq \CC$ is a set with $n$ elements, then
\[
    |f^{-1}(S)| \geq dn - d + 1.
\]
\end{lemma}

\begin{proof}
Let $e_p$ denote the ramification index of $f(x)$ at $p \in \CC$. Then for all $q \in \CC$,
\[
    d = \sum_{p \in f^{-1}(q)} e_p = |f^{-1}(q)| + \sum_{p \in f^{-1}(q)} (e_p - 1).
\]
A point $p \in \CC$ has $e_p > 1$ if and only if $p$ is a root of $f'(x)$ with multiplicity $e_p - 1$, hence
\[
    d - 1 = \sum_{p \in \CC} (e_p - 1).
\]
Thus,
\[
    |f^{-1}(S)| = \sum_{q \in S}|f^{-1}(q)| = dn - \sum_{p \in f^{-1}(S)} (e_p - 1) \geq dn - d + 1. \qedhere
\]
\end{proof}

\begin{prop}
\label{prop common preper}
Suppose that $f(x) \in \QQ[x]$ is a degree $d\geq 2$ polynomial such that
\[
    f([m]) \subseteq [n]
\]
for some integers $m > n \geq 0$. Then
\begin{enumerate}
    \item $f^{-1}([n]) \subseteq \bigcap_{i=0}^{m-n} \preper(f(x) + i,\CC)$,
    \item $|f^{-1}([n])| \geq dn - d + 1$, and
    \item $\preper(f(x) + i,\CC) \neq \preper(f(x) + j,\CC)$ for all $0\leq i < j \leq m-n$.
\end{enumerate}
Hence
\[
    d(n - 1) + 1 \leq \left|\bigcap_{i = 0}^{m-n} \preper(f(x) + i,\CC)\right| < \infty.
\]
\end{prop}

\begin{proof}
Let $0 \leq i \leq m -n$. If $f([m]) \subseteq [n]$, then
\[
    f(f^{-1}([n])) + i \subseteq [n + i]\subseteq [m] \subseteq f^{-1}([n]).
\]
Hence $f^{-1}([n]) \subseteq \preper(f(x) + i,\CC)$, proving (1). 
The lower bound $|f^{-1}([n])| \geq dn - d + 1$ follows immediately from Lemma \ref{lem: preimage} since $[n]$ contains $n$ points.

For (3) it suffices to prove that for any polynomial $h(x)$ and any positive integer $i$, $\preper(h(x),\CC) \neq \preper(h(x) + i,\CC)$. 
Since $h(x)$ is a polynomial, $\infty$ is a superattracting fixed point, and thus the set $\preper(h(x),\CC)$ of finite complex preperiodic points of $h(x)$ is bounded. 
Therefore, there exists some $q \in \preper(h(x),\CC)$ such that $q + i \not\in\preper(h(x),\CC)$.
Let $p \in h^{-1}(q)$.
Then $p \in \preper(h(x),\CC)$ by construction.
If $p \in \preper(h(x) + i,\CC)$, then $h(p) + i = q + i \in \preper(h(x)+i,\CC) \setminus \preper(h(x),\CC)$.
Otherwise, $p \in \preper(h(x),\CC) \setminus \preper(h(x) + i, \CC)$.
In either case, we have $\preper(h(x),\CC) \neq \preper(h(x)+i,\CC)$.

Finally, Baker and DeMarco \cite[Thm. 1.2]{BD} proved that if $f(x), g(x) \in \CC(x)$ are rational functions of degree at least 2, then $\preper(f(x),\CC) \neq \preper(g(x),\CC)$ implies $\preper(f(x),\CC) \cap \preper(g(x),\CC)$ is finite. Therefore,
\[
    d(n - 1) + 1 \leq \Big|\bigcap_{i = 0}^{m-n} \preper(f(x) + i,\CC)\Big| < \infty.\qedhere
\]
\end{proof}

\begin{example}
Consider the degree-6 polynomial
\[
    f(x) = \frac{x^6 - 45x^5 + 775x^4 - 6375x^3 + 25504x^2 - 45060x + 30960}{720}.
\]
One may check that
\[
    f([14]) \subseteq [10].
\]
Therefore Proposition \ref{prop common preper} implies that $\bigcap_{i=0}^4\preper(f(x) + i,\CC)$ is finite and contains at least 55 points.
\end{example}

Recall that $C_d$ for $d \geq 2$ is defined by
\[
    C_d := \sup_{f,g}|\preper(f,\CC) \cap \preper(g,\CC)|,
\]
where the supremum is taken over all $f(x), g(x) \in \CC[x]$ with $2 \leq \deg(f), \deg(g) \leq d$ such that $\preper(f,\CC) \neq \preper(g,\CC)$. 

\begin{thm}
\label{thm Cd lower bound}
Let $d\geq 2$ be an integer. There exists a polynomial $f_{d}(x) \in \QQ[x]$ with $2 \leq \deg(f_{d}) \leq d$ such that
\[
    \Big|\preper(f_{d}(x)+i,\CC) \cap \preper(f_{d}(x)+j,\CC)\Big| < \infty \text{ for all $0 \le i < j \le \log_2(d)$}
\]
and
\begin{equation}
\label{eqn common preper}
    \Big|\bigcap_{i=0}^{\lfloor \log_2(d)\rfloor} \preper(f_{d}(x) + i,\CC)\Big| \ge \deg(f_{d})(d - 1) + 1.
\end{equation}
Furthermore, there exist infinitely many
$d$ such that
\begin{equation}
\label{eqn Cd bound}
    C_d \geq d^2 + d\lfloor \log_2(d)\rfloor - 2d + 1.
\end{equation}
\end{thm}

\begin{proof}
Theorem \ref{thm main} implies that for $d\geq 2$ there exists a polynomial $f_{d}(x) \in \QQ[x]$ with $2 \leq \deg(f_{d}) \leq d$ such that
\[
    f_{d}([d + \lfloor\log_2(d)\rfloor]) \subseteq [d].
\]
Hence \eqref{eqn common preper} follows from Proposition \ref{prop common preper}.

Let $e_d := \lfloor \log_2(d)\rfloor$, then 
\[
    f_{d}([d+e_d]) \subseteq [d] \subseteq [d+e_d-1].
\]
Hence applying Proposition \ref{prop common preper} with $m = d + e_d$ and $n = d + e_d - 1$, we have
\[
    \deg(f_{d})(d + e_d - 2) + 1 \leq |\preper(f_{d}(x),\CC) \cap \preper(f_{d}(x)+1,\CC)| < \infty.
\]
Cohen, Shpilka, and Tal \cite[Cor. 1.2]{cohen} prove that for all $m$ sufficiently large, if $f(x) \in \QQ[x]$ is a polynomial with $\deg(f) \geq 2$ such that $f([m]) \subseteq [m-1]$, then
\[
    \deg(f) \geq m\Big(1 - \frac{4}{\log_2\log_2(m)}\Big).
\]
Hence, for all $d$ sufficiently large,
\[
    \deg(f_d) \geq (d + e_d)\Big(1 - \frac{4}{\log_2\log_2(d+e_d)}\Big)
    \geq d\Big(1 - \frac{4}{\log_2\log_2(d)}\Big).
\]
Note that the quantity on the right-hand side tends to $\infty$ with $d$; thus, the same is true for $\deg(f_d)$.

Now, let $d' \geq 2$, and let $d := \deg f_{d'} \le d'$. By the conclusion of the previous paragraph, there are infinitely many integers $d$ arising in this way. In this case, we have
\[
    C_{d} \ge d(d' + \lfloor \log_2(d)\rfloor - 2) + 1
        \geq d^2 + d\lfloor \log_2(d)\rfloor - 2d + 1.
\]
Thus \eqref{eqn Cd bound} holds for infinitely many $d$.
\end{proof}

\section{Examples of exceptional preperiodic behavior in every degree}
\label{sec d + 6}

In Theorem \ref{thm main} we showed that for all sufficiently large degrees $d$, there exists a degree-at-most-$d$ polynomial with at least $d + \lfloor \log_2(d)\rfloor$ rational preperiodic points.
However, the proof is not constructive, in the sense that it does not allow us to provide an explicit formula for such a polynomial.
In this section we construct a family of polynomials $r_d(x) \in \QQ[x]$ such that, for all $d\ge 2$, $r_d$ is a degree-$d$ polynomial with at least $d + 6$ rational preperiodic points. For
$d < 64$, this improves the lower bound on $B_d$ obtained from Theorem~\ref{thm main}.

First we introduce a doubly periodic sequence $\rho(m,d)$ and use its values to interpolate an auxiliary sequence of polynomials $s_d(x)$.

\begin{lemma}
\label{lem rho props}
There is a unique function $\rho : \ZZ^2 \to \ZZ$ satisfying the following properties:
\begin{enumerate}
    \item[(i)] $\rho(m + 3,d) = -\rho(m,d)$ for all $(m,d) \in \ZZ^2$,
    \item[(ii)] $\rho(m, d + 1) = -\rho(m+1,d)$ for all $(m,d) \in \ZZ^2$, and
    \item[(iii)] $\rho(0,0) = \rho(1,0) = 1$, and $\rho(2,0) = 0$.
\end{enumerate}
Furthermore, for all $(m,d) \in \ZZ^2$,
\begin{enumerate}
    \item $\rho(m,d+3) = \rho(m,d)$,
    \item $\rho(m+1,d+1) = \rho(m,d) + \rho(m,d+1)$,
    \item $\rho(m,d) = (-1)^d\rho(d + 1 - m, d)$.
\end{enumerate}
\end{lemma}

\begin{proof}
The initial values together with (i) imply that $\rho(m,0)$ is well-defined for all $m\in \ZZ$.
Then (ii) implies that $\rho(m,d) = (-1)^d\rho(m+d,0)$. 
Hence these three properties uniquely determine $\rho(m,d)$ for all $(m,d) \in \ZZ^2$.

\begin{table}[H]
    \centering
    \begin{tabular}{|c|c|c|c|c|c|c|}
    \hline
        $m$ & 0 & 1 & 2 & 3 & 4 & 5 \\
        \hline
        $\rho(m,0)$ & 1 & 1 & 0 & $-1$ & $-1$ & 0\\
    \hline
    \end{tabular}
    \caption{Initial values that determine $\rho(m,n)$}
    \label{tab:init vals}
\end{table}

(1) Properties (i) and (ii) imply that
\[
    \rho(m,d+3) = (-1)^3\rho(m+3,d) = -\rho(m+3,d) = \rho(m,d).
\]

(2) Since $\rho(m+6,d) = \rho(m,d)$, we can check the following identity by inspection:
\[
    \rho(m+2,0) = \rho(m+1,0) - \rho(m,0).
\]
Replacing $m$ by $m+d$ yields
\[
    \rho((m+1) + (d+1),0) = \rho(m + (d+1),0) - \rho(m + d,0),
\]
and repeatedly applying (ii) gives
\[
(-1)^{d+1}\rho(m+1,d+1) = (-1)^{d+1}\rho(m,d+1) + (-1)^{d+1}\rho(m,d);
\]
dividing by $(-1)^{d+1}$ yields (2).

(3) Using $\rho(m+6,d) = \rho(m,d)$ we may verify that for all $m \in \ZZ$,
\begin{equation}
\label{eqn symmetry row 0}
    \rho(m,0) = \rho(1 - m,0).
\end{equation}
Thus, if $d = 2n$ is even, then
\[
    \rho(m,2n) = \rho(m+2n,0)
    = \rho(m - 4n,0)
    = \rho(4n + 1 - m,0)
    = \rho(2n + 1 - m, 2n).
\]
Similarly, if $d = 2n + 1$ is odd, then \eqref{eqn symmetry row 0} implies that
\[
    \rho(m,2n+1) = -\rho(m+2n+1,0)
    =-\rho(m-4n+1,0)
    =-\rho(4n-m,0)
    =\rho(2n-1-m,2n+1).
\]
Finally, (i) implies
\[
    \rho(2n-1-m,2n+1)=-\rho(2n+2-m,2n+1).
\]
Hence in either case,
\[
    \rho(m,d) = (-1)^d\rho(d+1-m,d).\qedhere
\]
\end{proof}

Let $s_d(x) \in \QQ[x]$ be the unique degree-at-most-$d$ polynomial such that
\[
    s_d(m - \tfrac{d+1}{2}) = \rho(m,d),
\]
for $0\leq m \leq d$.
Lemma \ref{lem sd props} establishes some basic properties of $s_d(x)$.
Let $\delta$ denote the \emph{centered difference operator} defined by
\[
    \delta f(x) := f(x + \tfrac{1}{2}) - f(x - \tfrac{1}{2}).
\]

\begin{lemma}
\label{lem sd props}
Let $d\geq 0$.
\begin{enumerate}
    \item $s_d(-x) = (-1)^ds_d(x)$,
    \item $\delta s_{d+1}(x) = s_{d}(x)$,
    \item $\deg(s_d) = d$,
    \item $s_d(d+1 - \tfrac{d+1}{2}) = \rho(d+1,d)$,
    \item $s_d(d + 2 - \tfrac{d+1}{2}) = \rho(d+2,d) + 1$,
    \item $s_d(d + 3 - \tfrac{d+1}{2}) = \rho(d+3,d) + d + 2$.
\end{enumerate}
\end{lemma}

\begin{proof}
(1) Suppose $f(x) \in \RR[x]$ is a degree-at-most-$d$ polynomial.
Then $f(x) - (-1)^df(-x)$ has degree at most $d - 1$.
Hence if there are at least $\lfloor \tfrac{d+1}{2}\rfloor$ distinct pairs $\pm a$ of real numbers for which $f(a) = (-1)^df(-a)$, then it follows that the identity $f(x) = (-1)^df(-x)$ holds in $\RR[x]$.

Lemma \ref{lem rho props}(3) implies that for $1 \leq m \leq d$,
\[
    s_d(m - \tfrac{d+1}{2}) = \rho(m,d)
    = (-1)^d\rho(d + 1 - m, d)
    = (-1)^ds_d(d + 1 - m - \tfrac{d+1}{2})
    = (-1)^ds_d(\tfrac{d+1}{2} - m),
\]
since $1 \leq m \leq d$ is equivalent to $1\leq d + 1 - m \leq d$.
The set $\{m - \tfrac{d+1}{2} : 1 \leq m \leq d\}$ contains at least $\lfloor \tfrac{d+1}{2}\rfloor$ pairs $\pm a$, hence it follows that $s_d(x) = (-1)^ds_d(-x)$ for all $d\geq 0$.

(2) Let $f(x)$ be the degree-at-most-$d$ polynomial
\[
    f(x) := \det s_{d+1}(x) = s_{d+1}(x + \tfrac{1}{2}) - s_{d+1}(x - \tfrac{1}{2}).
\]
If $0\leq m \leq d$, then by Lemma \ref{lem rho props}(2),
\begin{align*}
    f(m- \tfrac{d+1}{2}) &= s_{d+1}(m + 1 - \tfrac{d+2}{2}) - s_{d+1}(m - \tfrac{d+2}{2})\\
    &= \rho(m+1,d+1) - \rho(m,d+1)\\
    &= \rho(m,d)\\
    &= s_d(m - \tfrac{d+1}{2}).
\end{align*}
Hence $f(x) = s_d(x)$.

(3) If $f(x)$ has degree $d\geq 1$, then $\delta f(x)$ has degree $d - 1$.
Since $s_0(x) = 1$ has degree 0 by construction, it follows from (2) that $\deg(s_d) = d$ for all $d\geq 0$.

(4) By (1) and Lemma \ref{lem rho props}(3),
\[
    s_d(d + 1 - \tfrac{d+1}{2}) = (-1)^ds_d(0 - \tfrac{d+1}{2}) 
    = (-1)^d\rho(0,d)
    = \rho(d+1,d).
\]

(5) We prove this identity by induction on $d$.
Since $s_0(x) = 1$ is constant, 
\[
    s_0(0 + 2 - \tfrac{0 + 1}{2}) = 1 = \rho(2,0) + 1.
\]
Now let $d\geq 1$ and suppose that the identity (5) holds for $d - 1$.
By (2), (4), and Lemma \ref{lem rho props}(2),
\begin{align*}
    s_d(d + 2 - \tfrac{d+1}{2}) &= s_d(d+1 - \tfrac{d+1}{2}) + s_{d-1}(d + 1 - \tfrac{d}{2})\\
    &= \rho(d+1,d) + \rho(d+1,d-1) + 1\\
    &= \rho(d+2,d) + 1.
\end{align*}

(6) Following the induction in (5) we first observe that
\[
    s_0(0 + 3 - \tfrac{0+1}{2}) = 1 = \rho(3,0) + 0 + 2.
\]
If $d\geq 1$ and we suppose that (6) holds for $d - 1$, then by (2), (5), and Lemma \ref{lem rho props}(2),
\begin{align*}
    s_d(d+3 - \tfrac{d+1}{2}) &= s_d(d+2 - \tfrac{d+1}{2}) + s_{d-1}(d+2 - \tfrac{d}{2})\\
    &= \rho(d+2,d) + 1 + \rho(d+2,d-1) + d - 1 + 2\\
    &= \rho(d+3,d) + d + 2.\qedhere
\end{align*}
\end{proof}

After a change of coordinates, the polynomials $s_d(x)$ provide explicit examples of a family of polynomial that compress many consecutive integers into an interval of fixed length.
These examples cover the low degree cases needed to complete the proof of Theorem \ref{thm main}.
\begin{cor}
\label{cor edge case family}
For $d\geq 0$ let $t_d(x) \in \QQ[x]$ be the degree $d$ polynomial defined by
\[
    t_d(x) := s_d(x - 2 - \tfrac{d+1}{2}) + 3.
\]
Then
\[
    t_d([d+4]) \subseteq [5].
\]
\end{cor}

\begin{proof}
Lemma \ref{lem sd props} and the definition of $s_d(x)$ implies that $s_d(m - \tfrac{d+1}{2}) \in [-2,2] \cap \ZZ$ for all integers $m$ such that $-2 \leq m \leq d + 2$. Hence it follows that $t_d([d+4]) \subseteq [5]$.
\end{proof}

\subsection{Constructing \texorpdfstring{$r_d(x)$}{rd(x)}}
\label{sec construct rd}
For $d\geq 0$, let $r_d(x) \in \QQ[x]$ be the polynomial sequence defined by
\[
    r_d(x) := \begin{cases} 
    s_d(x -3 - \tfrac{d+1}{2}) + 2 & \text{if $d$ is even,}\\ 
    s_d(x - 3 - \tfrac{d+ 1}{2}) - x + d + 6& \text{if $d$ is odd.}\end{cases} 
\]

\begin{thm}
\label{thm d+6}
For all $d\geq 2$, $r_d(x)$ is a degree-$d$, integer-valued polynomial such that if $d$ is even, then
\[
    r_d([d+6]) \subseteq [d+5],
\]
and if $d$ is odd, then
\[
    r_d([d+6]) \subseteq [d + 4].
\]
Thus,
\begin{enumerate}
    \item $B_d \geq d + 6$ for all $d\geq 2$, and
    \item $C_d \geq d^2 + 4d + 1$ for all $d\geq 2$.
\end{enumerate}
\end{thm}

\begin{proof}

Lemma~\ref{lem sd props}(3) implies that $r_d(x)$ has degree $d$, and the fact that $s_d(y - \frac{d+1}{2})$ is an integer for all $0 \le y \le d$, and therefore $r_d(x)$ is an integer for all $3 \le x \le d + 3$, implies that $r_d$ is integer-valued for all $d \ge 0$.

Suppose that $d\geq 2$ is even.
If $3\leq k \leq d + 4$, then the definition of $s_d(x)$ and Lemma \ref{lem sd props}(4) gives us,
\[
    r_d(k) = s_d(k - 3 - \tfrac{d+1}{2}) + 2 = \rho(k-3,d) + 2 \in [1,3].
\]
From $s_d(x)$ even we find that
\[
    r_d(d + 7 - x) = s_d(\tfrac{d+7}{2} - x) + 2 = s_d(x - \tfrac{d+7}{2}) + 2 = r_d(x).
\]
Hence by Lemma \ref{lem sd props}(5) and (6),
\begin{align*}
    r_d(1) &= r_d(d+6) = s_d(d + 3 - \tfrac{d+1}{2}) + 2 = \rho(d + 3,d) + d + 4 \in [d+3,d+5],\\
    r_d(2) &= r_d(d+5) = s_d(d + 2 - \tfrac{d+1}{2}) + 2 = \rho(d+2,d) + 3 \in [2,4].
\end{align*}
Thus if $d$ is even,
\[
    r_d([d+6]) \subseteq [d+5],
\]
from which it follows that $B_d \ge |\preper(r_d(x), \QQ)| \ge d + 6$.

Next suppose that $d \geq 2$ is odd. 
If $3 \leq k \leq d + 4$, then as above we have
\[
    r_d(k) = s_d(k - 3 - \tfrac{d+1}{2}) - k + d + 6 = \rho(k-3,d)-k+d+6 \in [1,d+4].
\]
Lemma \ref{lem sd props}(5) and (6) gives us
\begin{align*}
    r_d(d+5) &= s_d(d + 2 - \tfrac{d+1}{2}) + 1 = \rho(d+2,d) + 2 \in [1,3],\\
    r_d(d+6) &= s_d(d + 3 - \tfrac{d+1}{2}) = \rho(d+3,d) + d + 2 \in [d+1,d+3].
\end{align*}
Since $s_d(x)$ is odd,
\begin{align*}
    d + 5 - r_d(d + 7 - x) &= d + 5 - s_d((d + 7 - x) - 3 - \tfrac{d+1}{2}) + (d + 7 - x) - d - 6\\
    &= -s_d(-x + 3 + \tfrac{d+1}{2}) - x + d + 6\\
    &= s_d(x - 3 - \tfrac{d+1}{2}) - x + d + 6\\
    &= r_d(x).
\end{align*}
Thus,
\begin{align*}
    r_d(1) &= d + 5 - r_d(d+6) \in [2,4]\\
    r_d(2) &= d + 5 - r_d(d+5) \in [d+2,d+4].
\end{align*}
Therefore,
\[
    r_d([d+6]) \subseteq [d+4],
\]
and it follows that $B_d \geq d + 6$ for all $d\geq 2$.
Since for either parity of $d$ we have
\[
    r_d([d+6]) \subseteq [d + 5],
\]
Proposition \ref{prop common preper} implies that for all $d\geq 2$,
\[
    |\preper(r_d(x),\CC) \cap \preper(r_d(x)+1,\CC)| \geq d(d + 4) + 1 = d^2 + 4d + 1.
\]
Hence $C_d \geq d^2 + 4d + 1$ for all $d\geq 2$.
\end{proof}

Figure \ref{fig examples} illustrates the typical behavior of the polynomials $r_d(x)$ in the interval $[1,d+6]$.

\begin{figure}[H]
    \centering
    \includegraphics[scale=.5]{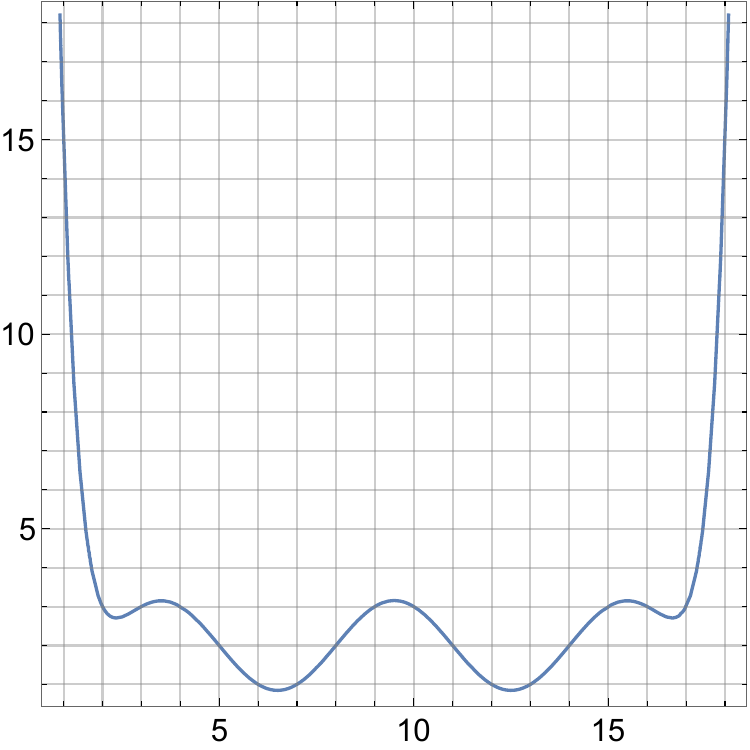}
    \hspace{.5in}
    \includegraphics[scale=.5]{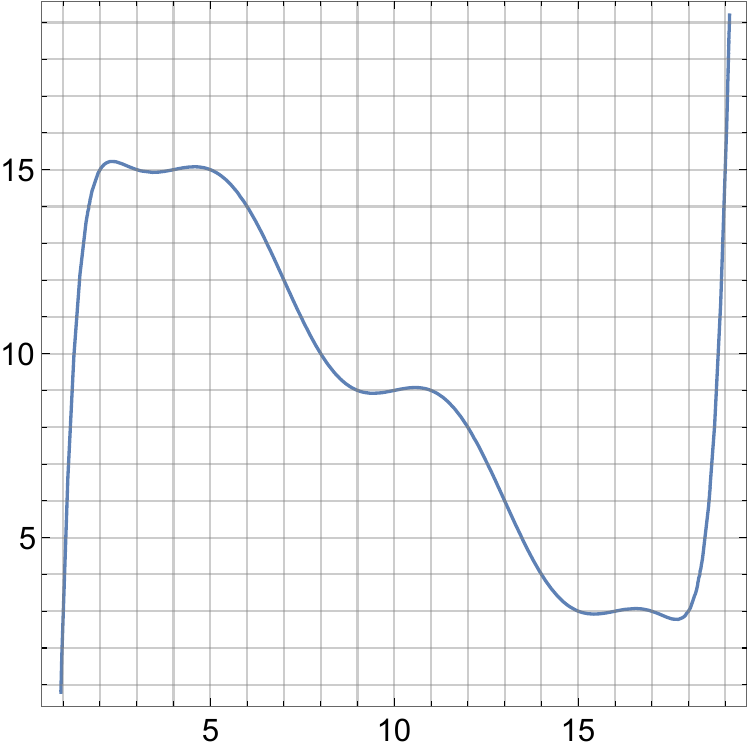}
    \caption{The graphs of $r_{12}(x)$ and $r_{13}(x)$}
    \label{fig examples}
\end{figure}

\subsection{Explicit formulas for \texorpdfstring{$s_d(x)$}{sd(x)}}
We suspect there is more of interest to say about the dynamical properties of the sequence of polynomials $r_d(x)$.
Thus to facilitate their future study we end this section by deriving explicit formulas for the polynomials $s_d(x)$, which then allow for direct calculation of $r_d(x)$.

For $d\geq 0$, let $c_d(x) \in \QQ[x]$ be the polynomial sequence defined by
\[
    c_{2k}(x) := \frac{1}{(2k)!}\prod_{j=1}^k \left(x^2 - \tfrac{(2j-1)^2}{4}\right)
    \hspace{.75in}
    c_{2k+1}(x) := \frac{1}{(2k+1)!}x\prod_{j=1}^k (x^2 - j^2).
\]
Note that, by construction, $c_d(x)$ is even when $d$ is even and $c_d(x)$ is odd when $d$ is odd.
A straightforward comparison of roots and leading coefficients implies that
\[
    \delta c_{d+1}(x) = c_{d}(x)
\]
for $d\geq 0$, and $\delta c_0(x) = \delta\,1 = 0$.
Furthermore, $c_{2k}(x)$ is integer-valued on $\ZZ + \tfrac{1}{2}$ and $c_{2k+1}(x)$ is integer-valued on $\ZZ$.

\begin{prop}
For all $k\geq 0$,
\[
    s_{2k}(x) := \sum_{j=0}^k (-1)^{k-j} c_{2j}(x)
    \qquad\text{and}\qquad
    s_{2k+1}(x) := \sum_{j=0}^k (-1)^{k-j} c_{2j + 1}(x).
\]
\end{prop}

\begin{proof}
First observe that $\{c_{2k} : k\geq 0\}$ forms a basis for the vector space 
of even polynomials in $\QQ[x]$. Therefore there are rational numbers $\alpha_j$ such that
\[
    s_{2k}(x) := \sum_{j=0}^k \alpha_j c_{2j}(x).
\]
Since 
\[
    c_{2k}(\tfrac{1}{2}) = \begin{cases} 1 & \text{if $k=0$,}\\ 0 & \text{otherwise,}\end{cases}
\]
it follows from Lemma \ref{lem rho props} and Lemma \ref{lem sd props} that 
\[
    \alpha_j = \delta^{2j} s_{2k}(\tfrac{1}{2}) 
    = s_{2(k-j)}(\tfrac{1}{2})
    = \rho(k-j,2(k-j))
    = \rho(3(k-j),0)
    = (-1)^{k-j}\rho(0,0)
    = (-1)^{k-j}.
\]
The expansion for $s_{2k+1}(x)$ follows by applying $\delta$ to the expansion of $s_{2k+2}(x)$.
\end{proof}

\section{Low degree examples}
\label{sec low degree}

In this final section we provide lower bounds for $B_d$ and $C_d$ in low degrees found by computation.
Table \ref{table low deg} gives examples of polynomials $f(x)$ and integers $m, n$ such that
\begin{equation}
\label{eqn compress}
    f([m]) \subseteq [n].
\end{equation}
These are the polynomials with the largest $m$ for each degree $d$ that we found by computer search.
Thus the entries in the column labeled $m$ also provide the best known lower bounds on $B_d$ for $2\leq d \leq 9$.

\begin{table}[H]
\begin{center}
\begin{tabular}{|c|c|c|c|}
\hline
$d$ & $m$ & $n$ & $f(x)$\\
\hline
2 & 8 & 7 & ${\scriptstyle\frac{x^2 -9x + 22}{2}}$\\
3 & 11 & 11 & ${\scriptstyle\frac{x^3 - 18x^2 + 89x - 66}{6}}$\\
4 & 10 & 8 & ${\scriptstyle\frac{x^4 - 22x^3 + 167x^2 - 506x + 552}{24}}$\\
5 & 13 & 9 & ${\scriptstyle\frac{x^5 - 35x^4 + 445x^3 - 2485x^2 + 5794x - 3600}{120}}$\\
6 & 14 & 10 & ${\scriptstyle\frac{x^6 - 45x^5 + 775x^4 - 6375x^3 + 25504x^2 - 45060x + 30960}{720}}$\\
7 & 15 & 15 & ${\scriptstyle\frac{x^7 - 56x^6 + 1246x^5 - 14000x^4 + 83629x^3 - 258104x^2 + 373764x - 151200}{5040}}$\\
8 & 16 & 16 & ${\scriptstyle\frac{x^8 - 68x^7 + 1918x^6 - 29036x^5 + 254989x^4 - 1309952x^3 + 3765012x^2 - 5343984x + 2862720}{20160}}$\\
8 & 16 & 15 & ${\scriptstyle\frac{x^8 - 68x^7 + 1946x^6 - 30464x^5 + 282569x^4 - 1559852x^3 + 4836124x^2 - 7320336x + 4273920}{40320}}$\\
9 & 19  & 17  & ${\scriptstyle\frac{x^9 - 90x^8 + 3426x^7 - 71820x^6 + 904449x^5 - 7002450x^4 + 32752124x^3 - 87183720x^2 + 116300160x - 55520640}{181440}}$\\
\hline
\end{tabular}
\end{center}
\caption{}
\label{table low deg}
\end{table}

One fast method to find these examples in low degree is to use the LLL basis reduction algorithm to find short vectors in the lattice $\Lambda_{d,e}$ defined in Section~\ref{sec lattices}.
Using this method we surveyed up to degree $d = 400$ and found examples giving $B_d \geq d + 8$ for all $11 \leq d \leq 283$ except for 
\[
     d \in \{21,
 219,
 221,
 235,
 237,
 241,
 244,\ldots, 247,
 249,
 251,
 255,\ldots, 266,
 268,
 269,
 271\}.
\]
For all other degrees $11 \leq d \leq 400$ the examples showed that $B_d \geq d + 6$, which we already knew by Theorem \ref{thm d+6}.

When $n < m$ in Table \ref{table low deg}, Proposition \ref{prop common preper} applied to $f([m]) \subseteq [m-1]$ allows us to extract lower bounds on $C_d$.
Note that the lower bound on $C_2$ in Table \ref{table cd bounds} comes from Example \ref{ex quadratic 2}.

In Table \ref{table cd bounds} we collect the best lower bounds on $C_d$ for small $d$ that we found through computational experiment. All of our examples came from common preperiodic points of $f(x)$ and $f(x) + 1$ for some polynomial $f(x)$ exhibiting dynamical compression.

\begin{table}[H]
    \centering
    \begin{tabular}{|c|cccccccccccccc|}
        \hline
        $d$ & 2 & 3 & 4 & 5 & 6 & 7 & 8 & 9 & 10 & 11 & 12 & 13 & 14 & 15\\\hline
        $C_d \ge $ & 26 & 27 & 40 & 60 & 78 & 84 & 120 & 162 & 190 & 198 & 228 & 260 & 294 & 330\\
        \hline
    \end{tabular}
    \caption{}
    \label{table cd bounds}
\end{table}

For degrees $d = 2, 4, 5, 6, 8, 9$, the polynomials giving the lower bound on $C_d$ come from Table \ref{table low deg}. For degrees $d = 3, 7$, the polynomials $r_d(x)$ constructed in Section \ref{sec construct rd} give the lower bound. For $10 \leq d \leq 15$, the polynomials $f_d(x)$ giving the lower bounds on $C_d$ are too large to print explicitly. Instead, Table \ref{table interp vals} lists $d + 1$ interpolating values $(f_d(1), f_d(2), \ldots, f_d(d+1))$ which uniquely determine the polynomial $f_d(x)$.

\begin{table}[H]
    \centering
    \begin{tabular}{|c|l|}
        \hline
        $d$ & $(f_d(1), f_d(2),\ldots, f_d(d+1))$\\
        \hline
        10 & (14, 6, 14, 6, 1, 6, 14, 17, 14, 10, 10)\\
        11 & (17, 1, 15, 3, 4, 14, 17, 12, 8, 9, 10, 6)\\
        12 & (17, 1, 17, 3, 4, 14, 17, 12, 6, 3, 3, 6, 12)\\
        13 & (17, 1, 17, 1, 3, 13, 16, 13, 10, 9, 9, 9, 8, 5)\\
        14 & (20, 4, 20, 4, 20, 14, 1, 8, 21, 18, 6, 6, 18, 21, 8)\\
        15 & (21, 1, 2, 20, 4, 1, 9, 10, 5, 3, 6, 11, 16, 19, 17, 12)\\
        \hline
    \end{tabular}
    \caption{}
    \label{table interp vals}
\end{table}

\end{document}